\documentclass[12pt,reqno]{amsart}
\usepackage{amsmath}
\usepackage{amssymb}
\usepackage{amstext}
\usepackage{a4wide}
\usepackage{lineno}

\allowdisplaybreaks \numberwithin{equation}{section}
\usepackage{color}

\numberwithin{equation}{section}

\newtheorem{theorem}{Theorem}[section]

\newtheorem{lemma}[theorem]{Lemma}

\theoremstyle{definition}

\newtheorem{definition}[theorem]{Definition}

\theoremstyle{remark}
\newtheorem{remark}[theorem]{Remark}

\begin{document}

\title[Nonlinear Stability of Planar Vortex Patches]
{Nonlinear Stability of planar vortex patches in an ideal fluid}

 \author{Daomin Cao and Guodong Wang}

\address{Institute of Applied Mathematics, Chinese Academy of Sciences, Beijing 100190, and University of Chinese Academy of Sciences, Beijing {\rm100049},  P.R. China}
\email{dmcao@amt.ac.cn}
\address{Institute for Advanced Study in Mathematics, Harbin Institute of Technology, Harbin {\rm150001}, P.R. China}
\email{wangguodong14@mails.ucas.ac.cn}


\begin{abstract}
In this paper, we prove nonlinear stability of planar vortex patches concentrated near an isolated minimum point of the Robin function in a general bounded domain. These vortex patches are stationary solutions of the two-dimensinal incompressible Euler equations. The result is obtained by showing that these concentrated vortex patches are in fact isolated maximizers of the kinetic energy among isovortical patches.
\end{abstract}

\maketitle

\section{Introduction}

In this paper, we consider the incompressible inviscid flow without external force in the plane, the  motion of which is governed by the following Euler equations:

\begin{equation}\label{1}
\begin{cases}
 \partial_t\mathbf{v}+(\mathbf{v}\cdot\nabla)\mathbf{v}=-\nabla P,\,\,x=(x^1,x^2)\in\mathbb R^2, t\in(0,+\infty),
\\ \nabla\cdot\mathbf{v}=0,
\end{cases}
\end{equation}
where $\mathbf{v}=(v_1,v_2)$ is the velocity field and $P$ is the scalar pressure. Here we assume that the fluid is of unit density.

Let $D\subset \mathbb{R}^2$ be a bounded and simply-connected domain with a smooth boundary, $\partial D$. When the fluid moves inside $D$, the impermeability boundary condition is usually imposed:
 \begin{equation}\label{2}
\mathbf{v}\cdot \mathbf{n}=0,
\end{equation}
where $\mathbf{n}$ is the outward unit normal of $\partial D$.
By introducing the vorticity function $\omega = \partial_{x^1} v_2-\partial_{x^2}v_1$ and using the identity $\frac{1}{2}\nabla|\mathbf{v}|^2=(\mathbf{v}\cdot\nabla)\mathbf{v}+J\mathbf{v}\omega$,
 the first equation of $\eqref{1}$ becomes
\begin{equation}\label{3}
 \partial_t\mathbf{v}+\nabla(\frac{1}{2}|\mathbf{v}|^2+P)-J\mathbf{v}\omega=0,
\end{equation}
where $J(v_1,v_2)=(v_2,-v_1)$ denotes clockwise rotation through ${\pi}/{2}$. Taking the curl on both sides of $\eqref{3}$ we get
\begin{equation}\label{499}
 \partial_t\omega+\mathbf{v}\cdot\nabla\omega=0.
\end{equation}

Since $\mathbf{v}$ is divergence-free and $D$ is simply-connected, by the Green formula $\mathbf{v}$ can be written as $\mathbf{v}=J\nabla\psi$ for some scalar function $\psi$. It is obvious that
\begin{equation}
 \begin{cases}
-\Delta \psi=\omega\text{ \quad \,\quad in $D$,}\\
\psi= \text{constant}  \text{\quad on $\partial D$.}
\end{cases}
\end{equation}
Without loss of generality, we can always assume that $\psi$ vanishes on $\partial D$ by adding a properly chosen constant. Therefore, $\psi$ can be expressed in terms of $\omega$ by
\[\psi(x)=G\omega(x):=\int_DG(x,y)\omega(y)dy,\]
 where $G$ is the Green function for $-\Delta$ in $D$ with zero Dirichlet
boundary condition, which has the following form
\[
G(x,y)=\frac{1}{2\pi}\ln \frac1{|x-y|}-h(x,y),   \,\,\,x,y\in
D.
\]

Taking into account the above relation between $\mathbf{v}$, $\omega$ and $\psi$, we are able to deduce the following equation satisfied by $\omega$:
\begin{equation}\label{599}
\partial_t\omega+J\nabla G\omega\cdot \nabla\omega=0,
\end{equation}
which is usually called the vorticity equation.
To deal with solutions with discontinuity, we need to interpret \eqref{599} in the weak sense.
\begin{definition}
We call $\omega\in L^\infty((0,+\infty)\times D)$ a weak solution to the vorticity equation \eqref{599} if
 \begin{equation}\label{997}
  \int_D\omega(x,0)\xi(x,0)dx+\int_0^{+\infty}\int_D\omega(\partial_t\xi+\nabla\xi\cdot J\nabla G\omega)dxdt=0,\,\,\, \forall \,\xi\in C_c^{\infty}([0,+\infty)\times D).
  \end{equation}
 \end{definition}
By standard elliptic regularity theory it is easy to check that  $\nabla G\omega\in L^\infty((0,+\infty)\times D)$, so the above definition makes sense.
By Yudovich \cite{Y}, for initial vorticity $\omega(\cdot,0)\in L^{\infty}(D)$, there is a unique weak solution $\omega$ to $\eqref{997}$, moreover,
 the distribution function of $\omega(\cdot,t)$ is independent of $t$, that is,
 \begin{equation}\label{112}
 |\{x\in D|\,\omega(x,t)>a\}|=|\{x\in D|\,\omega(x,0)>a\}|, \forall\, a \in\mathbb R, \,t\geq0,
 \end{equation}
where $|\cdot|$ denotes the two-dimensional Lebesgue measure.

For convenience, we also write $\omega(x,t)$ as $\omega_t(x)$. By $\eqref{112}$, if the initial vorticity $\omega_0$ has the form $\omega_0=\lambda I_{A_0}$, where $A_0\subset D$ is a measurable set, $\lambda\in \mathbb R$ represents the vorticity strength and $I_{A_0}$ denotes the characteristic function of $A_0$, i.e., $I_{A_0}(x)=1$ for $x\in A_0$ and $I_{A_0}=0$ elsewhere, then the evolved vorticity $\omega_t$ must be of the form $\omega_t=\lambda I_{A_t}$ with $|A_t|=|A_0|$ for all $t\geq 0$. We call such $\omega_t$ a vortex patch solution, or vortex patch briefly.

A weak solution of the vorticity equation is said to be steady if it does not depend on the time variable. Thus for any $\omega\in L^\infty(D)$, it is a steady solution if and only if
\begin{equation}\label{888}
\int_D\omega\nabla\xi\cdot J\nabla\psi dx=0,\,\forall\,\xi\in C^\infty_c(D).
\end{equation}
The search for dynamically possible steady solutions to the vorticity equation is an important and interesting problem in the study of two-dimensional incompressible Euler equations. It is easy to check that if $\omega\in L^\infty(D)$ satisfies
\[\omega=f(G\omega),\,\,\text{a.e. }x\in D,\]
where $f:\mathbb R\to\mathbb R$ is a Lipschitz continuous function, then $\omega$ must be a steady solution to the vorticity equation. Burton \cite{B3} proved that it is also true for any monotone function $f$.  As a special case, we immediately deduce that $\omega$ is a steady solution of the vorticity equation if it has the form
\begin{equation}\label{vp}
\omega=\lambda I_{\{G\omega>\mu\}},
\end{equation}
where $\lambda$ and $\mu$ are both positive constants. Here for simplicity we denote $\{x\in D|\,\, G\omega(x)>\mu\}$ by $\{G\omega>\mu\}$ and similar notations will be used in the sequel.





There are already several papers dealing with the existence of steady vortex patches of the form \eqref{vp}. See \cite{CPY}\cite{T} for example.
In this paper, we are mainly concerned with the nonlinear stability of such kind of steady vortex patches. Here by nonlinear stability we mean Liapunov type. To give the precise definition, let us first define $R_\omega$, the rearrangement class of $\omega$
\[
 R_\omega:=\{ w\in L^\infty(D)\, \mid |\{w>a\}|=|\{\omega>a\}| ,\,\,\forall\, a\in \mathbb R\}.
\]

\begin{definition}\label{def}
A steady vortex patch $\omega$ is called to be stable, if for any $\varepsilon >0$, there exists $\delta >0$, such that for any $\omega_0\in R_\omega$, $\|\omega_0-\omega\|_{L^1}<\delta$, there holds $\|\omega_t-\omega\|_{L^1}<\varepsilon$ for all $t\geq 0$, where $\omega_t(x)=\omega(x,t)$ is the solution of $\eqref{997}$ with initial vorticity $\omega_0$.
\end{definition}

In the above definition, we use the $L^1$ norm to measure the ``distance" between two solutions at any fixed time, which is very natural for vortex patch solutions. It should be noted that for vortex patch solutions the $L^1$ norm is equivalent to the $L^p$ norm for $p\in(1,+\infty).$

In this paper, we confine our attention to the nonlinear stability of steady vortex patch solutions with concentration property, that is, solutions satisfying
\begin{equation}\label{con}
\omega=\lambda I_{\{G\omega>\mu\}},\,\,\lambda|\{G\omega>\mu\}|=1,\,\,\{G\omega>\mu\}\subset B_{\delta}(x_0),
\end{equation}
where $\lambda>0$ is very large, $\delta>0$ is very small, and $x_0\in \overline{D}$ is a fixed point. It can be proved that if for any sufficiently large $\lambda$, there exists $\omega$ satisfying \eqref{con}, then $x_0$ is necessarily in the interior of $D$ and must be a critical point of $H$, the Robin function of $D$, defined by
\[H(x):=h(x,x).\]
See \cite{CGPY} for a rigorous proof.  On the other hand, if $x_0$ is a non-degenerate critical point of $H$, then for any sufficiently large $\lambda,$ there exists $\omega$ satisfying \eqref{con}. See \cite{CPY} for example.

The main result of this paper is as following.

\begin{theorem}\label{10000}

 Assume that $\omega_\lambda$ is a family of steady vortex patch solutions of the vorticity equation satisfying:
\begin{itemize}
\item[(i)] $\omega_\lambda=\lambda I_{\{G\omega_\lambda >\mu_\lambda\}} $, where $\mu_\lambda$ is a positive number depending on $\lambda$,

\item[(ii)]  $\int_D\omega_\lambda dx=1$,

\item[(iii)] $\{G\omega_\lambda >\mu_\lambda\}\subset B_{o(1)}(x_0),$ where $o(1)\to0$ as $\lambda\to+\infty$, and $x_0$ is an isolated minimum point and non-degenerate critical point of $H$.
    \end{itemize}
Then $\omega_\lambda$ is stable in the sense of Definition \ref{def} provided that $\lambda$ is large enough.
 \end{theorem}
 The proof of Theorem $\ref{10000}$ will be given in Section 3.

 \begin{remark}
Caffarelli and  Friedman \cite{CF1} proved that if $D$ is a convex domain, then the Robin function $H$ is strictly convex, thus in this case $H$ has a unique minimum point which is non-degenerate.
\end{remark}

\begin{remark}\label{855}
Theorem \ref{10000} is closely related to the vortex model(see\cite{MP2}, Chapter 4), which describes the motion of the fluid when the vorticity is sufficiently concentrated in $N$ small regions. In the case $N=1$, the vorticity is simplified as a Dirac measure called a point vortex, the location of which is determined by the following Kirchhoff-Routh equation:
\begin{equation}\label{299}
\begin{cases}
\frac{dx(t)}{dt}=-\frac{1}{2}J\nabla H(x(t)),\,\,t>0,\\
x(0)=x_0.
\end{cases}
\end{equation}
It is easy to see that
\begin{equation}\frac{dH(x(t))}{dt}=0,\,\,\forall\,t>0,
\end{equation}
that is, the point vortex moves along the level curve of $H$. If $x_0$ is critical point of $H$, then it is an equilibrium of the Kirchhoff-Routh equation, and the stability of this equilibrium is closely related to $D^2H(x_0)$, the Hessian of $H$ at $x_0$. In particular, if $D^2H(x_0)$ is positive definite(or equivalently, $x_0$ is an isolated minimum point and non-degenerate point), then $x_0$ must be stable. This can be proved by choosing $H$ as the Liapunov function. See \cite{MP2}, Chapter 3 for example.
 In such a way, Theorem \ref{10000} can be interpreted as a desingularized version of the stability for the vortex model.
\end{remark}

The analysis of stability for steady Euler flows in two dimensions is very important in fluid mechanics and has been studied by many authors in history. See for example \cite{A,A2,B3,K,La,Lo,Ta,WP}. Here we recall some of the relevant and significant results associated with the stability of planar vortex patches. The first stability result is due to Kelvin in \cite{K}, where he established the linear stability for circular vortex patches in the whole plane. Later Love \cite{Lo} proved linear stability of a rotating Kirchhoff elliptical vortex patch in the plane. Another excellent work is due to Arnold \cite{A,A2}. Arnold gave several criteria for nonlinear stability of smooth steady Euler flows in general bounded domains, which can be seen as a nonlinear version of the classical Rayleigh inflection point criterion for linear stability of shear flows in a channel. Moreover, he asserted in \cite{A3} in general terms that a steady flow can be seen as a constrained critical point of the kinetic energy; if this critical point is a non-degenerate extreme, then it should be stable. Unfortunately, it seems that his method is not easy to apply to prove the nonlinear stability of vortex patches,  since there is strong discontinuity for the vorticity.

In 1985, based on energy conservation, Wan and Pulvirenti in \cite{WP} proved nonlinear stability of circular vortex patches in an open disk. In that paper, they established a relative variational principle for the kinetic energy and turned the $L^1$ perturbation problem into a $C^1$ perturbation problem. When dealing with the $C^1$ perturbation case, the key ingredient of their proof is that for a circle, the Green function is explicitly known and rotationally invariant, then the $C^1$ perturbation case can be handled by spectral analysis of a negative definite operator. In 1987, Tang \cite{Ta} proved nonlinear stability of both circular vortex patches and rotating elliptic vortex patches in the plane based on the same idea as in \cite{WP}. The method developed by Marchioro, Pulvirenti and Tang is still not easy to apply to prove nonlinear stability of steady vortex patches in general bounded domains, since in such cases the Green functions have no explicit expression and good symmetry anymore. However, their idea of turning $L^1$ perturbation into $C^1$ perturbation is very clever, and is used in the present paper as a key ingredient to prove Theorem \ref{10000}.

In 2005, Burton in \cite{B3} proved that a steady vortex flow as an isolated maximizer of the kinetic energy relative to an ``isovortical surface"(rearrangement class of a given function) is stable in some $L^p$ norm. But only in few cases can the isolatedness of the maximizer of the kinetic energy be verified. Burton's result is crucial for the proof of Theorem \ref{10000} and will be stated precisely in Section 2.

Recently Cao, Guo, Peng and Yan in \cite{CGPY} established a local uniqueness result for steady vortex patches concentrated near a non-degenerate critical point of the Robin function. They considered the following semilinear elliptic equation satisfied by the stream function:
 \begin{equation}\label{44}
 \begin{cases}
-\Delta \psi_\lambda=\lambda I_{\{\psi_\lambda>\mu_\lambda\}}\text{ \quad in $D$},\\
\psi_\lambda= 0\quad\quad\quad\quad\,\,\,\,\,\, \text{\quad on $\partial D$},
\end{cases}
\end{equation}
where $\mu_\lambda>0$ is an unknown constant and $\lambda|\{\psi_\lambda>\mu_\lambda\}|=1$. They proved that if the vortex core $\{\psi_\lambda>\mu_\lambda\}$ shrinks to a non-degenerate critical point of the Robin function as $\lambda\rightarrow +\infty$, then the solution of $\eqref{44}$ is unique provided that $\lambda$ is large enough. The precise statement will be given in Section 2. This local uniqueness result is used in this paper to verify the isolatedness of the maximizer of the kinetic energy.

  This paper is organized as follows. In Section 2 we state several known results that will be used later. In Section 3 we give the proof of Theorem \ref{10000}. In Section 4 we give the construction of steady vortex patches concentrated near a given isolated minimum point of the Robin function for completeness.

\section{Preliminaries}

In this section, we mainly recall several known results on the existence, local uniqueness and stability of planar vortex patches in general bounded domains.

Let us set some notations first. For any function $f:D\to \mathbb R$, we denote $supp(f):=\overline{\{f\neq0\}}$. For any planar set $A$, we use $diam(A)$ to denote the diameter of $A$, i.e.,
\[diam(A)=\sup_{x,y\in A}|x-y|.\]
Define $E(\omega)$ as the kinetic energy of the fluid with vorticity $\omega$ by setting
\begin{equation}
E(\omega):=\frac{1}{2}\int_D \int_D G(x,y)\omega(x)\omega(y)dxdy,
\end{equation}
By integration by parts, it is easy to check that
\begin{equation}
E(\omega)=\frac{1}{2}\int_D\psi(x)\omega(x)dx=\frac{1}{2}\int_D |\nabla G\omega(x)|^2dx.
\end{equation}

\subsection{Existence of Steady Vortex Patches}

To our knowledge, there are mainly two methods dealing with the existence of steady vortex flows in two dimensions. The first one is called the stream function method, whose starting point is to solve \eqref{44} for a give nonlinearity $f$. See \cite{BP,CPY,CLW,N,SV} for example. The other one is to solve a variational problem for the kinetic energy subject to some appropriate constraints for the vorticity. See \cite{Ba,BF,B1,B2,B4,T} for example.

In this subsection, for our purpose, we recall the existence of steady vortex patches of the form \eqref{vp} via the vorticity method. The result and idea are mostly based on Turkington \cite{T}.

To this end, we define \[K_{\lambda}\triangleq\{\omega\in L^{\infty}(D)|\,\,0\leq \omega\leq \lambda,\,\int_{D}\omega(x)dx=1\},\] where $\lambda$ is positive and large enough so that $K_\lambda$ is not empty.
Let $\varepsilon$ be the positive number determined by $\lambda\pi\varepsilon^2=1$.

Turkington \cite{T} proved the following result.

\begin{theorem}\label{21}
$E$ attains it maximum on $K_\lambda$, and each maximizer $\omega_\lambda$ has the form $\omega_{\lambda}=\lambda I_{\{G\omega_{\lambda}>\mu_\lambda\}}$(thus must be a steady solution to the vorticity equation), where $\mu_\lambda$ is a real number depending on ${\lambda}$. Moreover, $supp(\omega_\lambda)$ shrinks to a global minimum point of $H$ as $\lambda \to +\infty$; more precisely, $diam(supp(\omega_{\lambda}))\leq C\varepsilon$ and $\int_{D}x\omega(x)dx\rightarrow x_0$ as $\lambda \to +\infty$, where $C$ is a positive number not depending on $\lambda$ and $x_0$ is a global minimum point of $H$.
\end{theorem}

\begin{remark}
It should be noted that $\lim_{x\to\partial D}H(x)=+\infty,$ so $H$ attains its global minimum in $D$. But there may be more than one global minimum point of $H$, and we do not know which one $supp(\omega_\lambda)$ shrinks to as $\lambda\to +\infty$.
\end{remark}

 By performing a similar procedure to the one in \cite{T}, in the following theorem we prove existence of a family of steady vortex patches concentrated near an isolated minimum point of $H$. Similar results have been obtained by Elcrat and Miller \cite{EM,EM2}.

 To state our result, we need some notations slightly different. Here and in the sequel let $x_1\in D$ be a strict local minimum point of $H$. We choose a sufficiently small positive number $r$ such that $x_1$ is the unique minimum point of $H$ on $\overline{B_{r}(x_1)}$
 and $B_r(x_1)\subset\subset D$. Define
 \[N_{\lambda}\triangleq\{\omega\in L^{\infty}(D)|\,\,0\leq \omega\leq \lambda,\,\int_{D}\omega(x)dx=1,\,supp(\omega)\subset B_r(x_1)\}.\]

\begin{theorem}\label{22}
$E$ attains it maximum on $N_\lambda$, and each maximizer $\omega_\lambda$ has the form $\omega_{\lambda}=\lambda I_{\{G\omega_{\lambda}>\mu_\lambda\}\cap B_r(x_1)}$, where $\mu_\lambda$ is a real number depending on ${\lambda}$. Moreover, $supp(\omega_{\lambda})$ shrinks to $x_1$ as $\lambda \rightarrow +\infty$; more precisely, diam(supp($\omega_{\lambda}))$ $\leq$C$\varepsilon$ for some $C$ independent of $\lambda$ and $\int_{D}x\omega_\lambda(x)dx\rightarrow x_1$ as $\lambda \rightarrow +\infty$. If $\lambda$ is large enough, then $\omega_\lambda$ is a steady solution to the vorticity equation, that is, $\omega_\lambda$ satisfies \eqref{888}.
\end{theorem}
For readers' convenience, we will give the complete proof of Theorem \ref{22} in Section 4.

\begin{remark}\label{46}
In fact, we will prove in Section 4 that $\mu_\lambda\to+\infty$ as $\lambda\to+\infty.$ Taking into account the fact that $supp(\omega_{\lambda})$ shrinks to $x_1$, by using maximum principle we immediately deduce that $\omega_\lambda$ has the form
$\omega_{\lambda}=\lambda I_{\{G\omega_{\lambda}>\mu_\lambda\}}$ if $\lambda$ is sufficiently large.
\end{remark}

\subsection{Burton's Stability Criterion}

 In \cite{B3}, Burton proved that steady vortex flow as isolated maximizer of the kinetic energy on an isovortical surface is stable in some $L^p$ norm. In the case of vortex patches, the corresponding result can be stated as follows:

\begin{theorem}\label{26}
 Let $\omega$ be a vortex patch in $D$(that is, $\omega=\lambda I_{A}$ for some $\lambda>0$ and $A\subset D$). Suppose that $\omega$ is an isolated maximizer of the kinetic energy $E$ on $R_\omega$, that is, there exists $\delta_0>0$, such that for any $\bar{\omega}\in R_\omega$, $0<\|\bar{\omega}-\omega\|_{L^1}<\delta_0$, there holds $E(\bar{\omega})<E(\omega)$. Then $\omega$ is stable in the sense of Definition \ref{def}.
\end{theorem}

\begin{remark}
Burton's result is a very general stability criterion. According to \cite{B3}, Theorem 5, any steady vortex flow as isolated maximizer or minimizer relative to the rearrangement class of a given $L^p$ function with $p\geq3/2$ is stable in the $L^p$ norm. But in the case of vortex patches, Theorem $\ref{26}$ is enough for our use.
\end{remark}

The key assumption in Theorem $\ref{26}$ is the ``isolatedness'' of the energy maximizer, which in general terms is an non-degeneracy condition. In \cite{B3}, only one example of strict global maximizer is given, i.e., $D$ is an open disk and $\omega$ is a non-negative radially symmetric decreasing function. For general maximizers, especially local maximizers, the isolatedness assumption is not easy to be verified, since ``isolatedness"  is equivalent to uniqueness in some sense, which is usually a more difficult problem than existence.

\subsection{Local Uniqueness of Steady Vortex Patches}
Fortunately, Cao et al. \cite{CGPY} proved the following local uniqueness result of steady vortex patches near a non-degenerate critical of the Robin function based on fine estimates for the corresponding stream function, which may be used to prove isolatedness of the energy maximizer on isovortical patches.

\begin{theorem}\label{27}
Let $x^*$ be a non-degenerate critical point of $H$. Then there exists $\lambda_0>0,$ such that for each $\lambda>\lambda_0,$ the $\omega_\lambda$ satisfying the following properties is unique:
\begin{itemize}
\item[(i)] $\omega_\lambda=\lambda I_{\{G\omega_\lambda >\mu_\lambda\}}$, where $\mu_\lambda$ is an unknown positive number depending on $\lambda$;

\item[(ii)]  $\int_D\omega_\lambda dx=1$;

\item[(iii)] the support of $\omega_\lambda$ shrinks to $x^*$, or equivalently, $supp(\omega_\lambda)\subset B_{o(1)}(x^*)$.
\end{itemize}
 \end{theorem}


\begin{remark}
When $\lambda$ is not large, there is no uniqueness result on steady vortex patches in general bounded domains.
 However, we conjecture that this is true when $D$ is a convex domain.
\end{remark}

\section{Proof of the Theorem \ref{10000}}
In this section, we give the proof of Theorem \ref{10000}.

When $D$ is a convex domain, there is no gap between local uniqueness and isolatedness of the energy maximizer, so in this case the proof of Theorem \ref{10000} is much simpler. For clarity, we give the short proof here.

\begin{proof}[Proof of Theorem \ref{10000}(the case $D$ is convex)]

By Theorem \ref{26}, it suffices to show that $\omega_\lambda$ is an isolated maximizer of $E$ over $R_{\omega_\lambda}.$

When $D$ is convex, there is only one critical point of $H$, which is exactly the unique global minimum point. So the $x_0$ in Theorem \ref{10000} must be the unique minimum point of $H$.
Now combining Theorem \ref{22} and Theorem \ref{27}, we immediately deduce that the $\omega_\lambda$ in Theorem \ref{10000} must be the uniqueness maximizer of $E$ over $K_\lambda$ if $\lambda$ is large enough. Taking into the fact that $R_{\omega_\lambda}\subset K_\lambda$, we reach the conclusion that $\omega$ is in fact the unique maximizer of $E$ over $R_{\omega_\lambda}$, which is the desired result.

 \end{proof}

When $D$ is a general bounded domain, there is a gap between local uniqueness and isolatedness of energy maximizer. To eliminate the gap, we follow the idea of Wan and Pulvirenti \cite{WP} to turn $L^1$ perturbation into $C^1$ perturbation.

To make the idea more adaptable, we give a more general stability criterion for steady vortex patches, that is, Lemma $\ref{31}$ below. For convenience, we introduce the following notation. Let $\omega$ be a vortex patch enclosed by a $C^1$ closed curve denoted by $\gamma_\omega$(thus $\gamma_\omega$ is a planar set). A $\delta$ neighbourhood of $\gamma_\omega$ is defined by:
  \begin{equation}
  \gamma_{\omega,\delta}\triangleq\{x\in \mathbb R^2\,\,|\,\,  dist(x,\gamma_\omega)<\delta\}.
\end{equation}

\begin{lemma}\label{31}
Suppose that $\omega_0$ is a steady vortex patch in $D$ satisfying the following conditions:

(C1) $\omega_0$ has the form $\omega_0=\lambda I_{\{\psi_0>\mu\}}$ for some $\mu>0$, where $\psi_0=\int_DG(x,y)\omega_0(y)dy$;

(C2) $\partial\{\psi_0>\mu\}$ is a $C^1$ closed curve and $\frac{\partial \psi_0}{\partial\vec{n}}<0$ on this curve, where $\vec{n}$ is the outward unit normal of $\partial\{\psi_0>\mu\}$;


(C3)  there exists $\delta>0$, such that if $\omega_1\in R_{\omega_0}$ is another vortex patch(not necessarily steady) enclosed by a $C^1$ simple curve $\gamma_{\omega_1}$ and $\gamma_{\omega_1}\subset\gamma_{\omega_0,\delta}$, then $E(\omega_1)\leq E(\omega_0)$, the equality holds if and only if $\omega_0\equiv\omega_1$.\\
Then $\omega_0$ is stable.

 \end{lemma}

 \begin{proof}

By Theorem $\ref{26}$, it suffices to show that $\omega_0$ is an isolated maximizer of $E$ over $R_{\omega_0}$. We show this by contradiction in the following.

 Suppose that $\omega_0$ is not an isolated maximizer of $E$ over $R_{\omega_0}$, then we can choose a sequence $\{\omega_n\}\subset R_{\omega_0}$ such that $0<\|\omega_n-\omega_{0}\|_{L^1}<\frac{1}{n}$, and
  \begin{equation} \label{99}E(\omega_n)\geq E(\omega_0).\end{equation}

For such a sequence, we have the following claim:

{\bf{Claim}}: Denote $\psi_n=G\omega_n$ and $\psi_0=G\omega_0$, then, if $n$ is large, there exists a unique $\nu_n>0$ such that

(i) $\partial\{\psi_n>\nu_n\}$ is a $C^1$ closed curve,

(ii) $|\{\psi_n>\nu_n\}|=|\{\psi_0>\mu\}|$,

(iii)$\partial\{\psi_n>\nu_n\} \subset \gamma_{\omega_0,\delta}$, where $\delta$ is the one in $(C3)$.
\[\]

{\bf{Proof of the claim}}:
First, notice that $\psi_n$ and $\psi_0$ satisfy the following equations:
\begin{equation}
\begin{cases}
-\Delta\psi_n=\omega_n,
 \\ -\Delta\psi_0=\omega_0.
\end{cases}
\end{equation}
For vortex patches, $\|\omega_n-\omega_0\|_{L^1}\rightarrow 0$ implies $\|\omega_n-\omega_0\|_{L^p}\rightarrow 0$ for any $1\leq p<+\infty$. Then by $L^p$ estimates we have $\|\psi_n-\psi_0\|_{W^{2,p}}\rightarrow 0$ for any $1<p<+\infty$.
Choosing $p$ large enough, by the Sobolev embedding $W^{2,p}(D)\hookrightarrow C^{1,\alpha}(\overline{D})$ for some $\alpha\in(0,1)$, we obtain
\begin{equation}\label{333}
\|\psi_n-\psi_0\|_{C^1}\rightarrow 0.
\end{equation}

By $(C2)$ we can take $\delta_0>0$ small, such that the set $\{\mu-\delta_0<\psi_0<\mu+\delta_0\}$ is an annulus-like domain and $\frac{\partial\psi_0}{\partial \vec{n}}<0$ on each closed curve $\{\psi_0=a\}, \mu-\delta_0\leq a\leq\mu+\delta_0$.
This is true by the continuity of $\psi_0$ and $\nabla\psi_0$.  Since $\|\psi_n-\psi_0\|_{C^1}\rightarrow 0$, we have  $\frac{\partial\psi_n}{\partial \vec{n}}<0$ on each curve $\{\psi_0=a\}, \mu-\delta_0\leq a<\mu+\delta_0$  if $n$ is large enough. Thus $\nabla\psi_n\neq0$ in the annulus-like domain $\{\mu-\delta_0\leq\psi_0\leq\mu+\delta_0\}$.

Now choose $\varepsilon<\delta_0$ small and define $M:=\max_{\{\psi_0=\mu-\varepsilon\}}\psi_n, m := \min_{\{\psi_0=\mu+\varepsilon\}}\psi_n$. By the implicit function theorem, $\{\psi_n=M\}$ and $\{\psi_n=m\}$ are both $C^1$ curves locally. By the properties that $\psi_n$ strictly increases along the direction $\nabla\psi_0$ in the annulus-domain $\{\mu-\delta_0\leq \psi_0\leq\mu+\delta_0\}$, the curve $\{\psi_n=M\}$ can not go outside $\{\psi_0\geq\mu-\varepsilon\}$ .

In fact, suppose that there exists $x_0\in \{\psi_0<\mu-\varepsilon\}\cap\{\mu-\delta_0<\psi_0<\mu+\delta_0\}$ and $\psi_n(x_0)=M$, we can find $x_1\in \{\psi_0=\mu-\varepsilon\}$ by solving the following ODE:
\begin{equation}
\begin{cases}
\frac{dx(t)}{dt}=-\nabla\psi_0(x),\\
 x(0)=x_0.
\end{cases}
\end{equation}
Since $\frac{\partial\psi_n}{\partial \vec{n}}<0$, we have $\psi_n(x_1)>\psi_n(x_0)$. But by the definition of $M$, $\psi_n(x_1)\leq M=\psi_n(x_0)$, which is a contradiction.

On the other hand, by taking $n$ large enough the curve $\{\psi_n=M\}$ can not enter $\{\psi_0>\mu\}$. In fact, suppose that there exists $x_0$ such that $x_0\in \{\psi_0=\mu\}$ and $x_0\in \{\psi_n=M\}$, then we have $\psi_n(x_0)=M$. But by $\eqref{333}$ $\psi_n(x_0)\rightarrow \mu$ and $M\rightarrow \mu-\varepsilon$ as $n\rightarrow+\infty$, from which we deduce that $\psi_n(x_0)>M$ if $n$ is large enough, which is a contradiction.

Therefore, $\{\psi_n=M\}\cap\{\mu-\delta_0<\psi_0<\mu+\delta_0\}$ must be a $C^1$ closed curve and

\begin{equation}
\{\psi_0>\mu\}\subset \{\psi_n>M\}\cap \{\mu-\delta_0<\psi_0<\mu+\delta_0\}   \subset\{\psi_0>\mu-\varepsilon\}.
\end{equation}
Similarly $\{\psi_n=m\}\cap \{\mu-\delta_0<\psi_0<\mu+\delta_0\}$ must be a $C^1$ closed curve and

\begin{equation}
\{\psi_0>\mu+\varepsilon\}\subset \{\psi_n>m\}\cap \{\mu-\delta_0<\psi_0<\mu+\delta_0\}   \subset\{\psi_0>\mu\}.
\end{equation}
Hence
\begin{equation}
\begin{split}
|\{\psi_n>m\}\cap \{\mu-\delta_0<\psi_0<\mu+\delta_0\}|  & \leq|\{\psi_0>\mu\}|\\
&\leq| \{\psi_n>M\}\cap \{\mu-\delta_0<\psi_0<\mu+\delta_0\}|.
\end{split}
\end{equation}

By the continuity of $\psi_n$ we can choose $\nu_n\in[m,M]$ such that \[|\{\psi_n>\nu_n\}\cap\{\mu-\delta_0<\psi_0<\mu+\delta_0\}|=|\{\psi_0>\mu\}|,\]
and \[\{\psi_n=\nu_n\}\cap\{\mu-\delta_0<\psi_0<\mu+\delta_0\}\subset \{\mu-\varepsilon<\psi_0<\mu+\varepsilon\}.\]
 Note that such $\nu_n$ must be unique because $\frac{\partial\psi_n(x)}{\partial\vec{n}}<0$ on the curve $\{\psi_n=\nu_n\}$.
By $\eqref{333}$ $\sup_{\{\psi_0<\mu-\delta_0\}}\psi_n<\nu_n$ if $n$ is large enough and by strong maximum principle $\inf_{\{\psi_0>\mu+\delta_0\}}\psi_n>\nu_n$, so $\{\psi_n>\nu_n\}\cap\{\mu-\delta_0<\psi_0<\mu+\delta_0\}=\{\psi_n>\nu_n\}$ if $n$ is large enough.

That is, for any $\varepsilon>0$, if $n$ is large, we can choose a unique $\nu_n$, such that $\partial \{\psi_n>\nu_n\}=\{\psi_n=\nu_n\}$ is a $C^1$ closed curve, moreover, $|\{\psi_n>\nu_n\}|=|\{\psi_0>\mu\}|$ and $\{\psi_n=\nu_n\}\subset \{\mu-\varepsilon<\psi_0<\mu+\varepsilon\}$. Hence the claim is proved.
\[\]

Now we continue our proof of Lemma $\ref{31}$.
Define $\bar{\omega}_n=\lambda I_{\{\psi_n>\nu_n\}}$, where $\nu_n$ is the one chosen in the above Claim. It is obvious that $\bar{\omega}_n\in R_{\omega_0}$. Now we compare $E(\bar{\omega}_n)$ and $E(\omega_n)$ as follows:
\begin{equation}\label{rem1}
\begin{split}
E(\bar{\omega}_n)-E(\omega_n)
=&\frac{1}{2}\int_D  \bar{\omega}_n\bar{\psi_n} dx-\frac{1}{2}\int_D \omega_n\psi_n dx\\
=&\int_D\psi_n(\bar{\omega}_n-\omega_n) dx+\frac{1}{2}\int_D(\bar{\psi}_n-\psi_n)(\bar{\omega}_n-\omega_n) dx, \\
\end{split}
\end{equation}
where we used $\int_D\psi_n\bar{\omega}_n dx=\int_D\bar{\psi}_n\omega_n dx$ by the symmetry of the Green function. By integration by parts we have
\[\frac{1}{2}\int_D(\bar{\psi}_n-\psi_n)(\bar{\omega}_n-\omega_n) dx=\frac{1}{2}\int_D|\nabla(\bar{\psi}_n-\psi_n)|^2dx,\]
therefore we get
\[E(\bar{\omega}_n)-E(\omega_n)\geq\int_D\psi_n(\bar{\omega}_n-\omega_n)dx+\frac{1}{2}\int_D|\nabla(\bar{\psi}_n-\psi_n)|^2dx.\]

Since $|\{\psi_n>\nu_n\}|=|\{\psi_0>\mu\}|$ and $\bar{\omega}_n=\lambda I_{\{\psi_n>\nu_n\}}$, the integral $\int_D\psi_n\omega_ndx$ attains its maximum if and only if $\omega_n=\bar{\omega}_n$, thus we obtain
 \begin{equation}\label{rem2}
\int_D\psi_n(\bar{\omega}_n-\omega_n)dx\geq 0.
\end{equation}
Combining \eqref{rem1} and \eqref{rem2}, we get
\begin{equation}\label{100}E(\bar{\omega}_n)\geq E(\omega_n),\end{equation}
and the equality holds if and only if $\omega_n\equiv\bar{\omega}_n$. By $\eqref{99}$ and $\eqref{100}$ we have
\[E(\bar{\omega}_n)\geq E(\omega_0).\]

On the other hand, by (iii) in the above Claim we can take $n$ large enough such that $\gamma_{\bar{\omega}_n}\subset\gamma_{\omega_0,\delta}$, then by $(C3)$ we have
\[\bar{\omega}_n\equiv\omega_0,\]
which implies that
\[E(\bar{\omega}_n)=E(\omega_0)=E(\omega_n),\]
hence $\omega_n\equiv\bar{\omega}_n\equiv\omega_0$. This leads to a contradiction since $|\omega_n-\omega_0|_{L^1}>0$ for each $n$.
Therefore Lemma 3.1 is proved.

 \end{proof}

Lemma 3.1 is a general stability criterion which does not require the vortex patch to be concentrated. However, $(C1)-(C3)$ are not easy to be verified in general.
To continue, we need the following result from \cite{CGPY}.

\begin{lemma}[\cite{CGPY}]\label{32}
Let $\omega_\lambda$ be the steady vortex patch in Theorem $\ref{10000}$. Then $\{G\omega_\lambda>\mu_\lambda\}$ is a simply-connected domain with a $C^1$ boundary, and $\frac{\partial G\omega_\lambda}{\partial \vec{n}}<0$ on this boundary, provided that $\lambda$ is large enough.
\end{lemma}

Having made the above preparations, now we are ready to prove Theorem $\ref{10000}$.
\begin{proof}[Proof of Theorem \ref{10000}(general case)]  By Lemma $\ref{31}$, it suffices to show that $\omega_\lambda$ satisfies $(C1)$-$(C3)$ in Lemma $\ref{31}$. By Lemma $\ref{32}$, $(C1)$ and $(C2)$ holds true. We need only to verify $(C3)$.

By Theorem \ref{27}, the steady vortex patch in Theorem $\ref{10000}$ is the same as the one in Theorem $\ref{22}$ when $\lambda$ is large enough, so we need just to verify $(C3)$ for the $\omega_\lambda$ in Theorem $\ref{22}$.

Suppose that there exists another vortex patch $\bar{\omega}_\lambda$ enclosed by a $C^1$ closed curve $\gamma_{\bar{\omega}_\lambda}$, $\gamma_{\bar{\omega}_\lambda}\subset\gamma_{\omega_\lambda,\delta}$ and $E(\omega_\lambda)= E(\bar{\omega}_\lambda)$. To finish the proof, it suffices to show that $\omega_\lambda\equiv\bar{\omega}_\lambda$. Without loss of generality, we assume that $\delta$ is sufficiently small such that $\bar{\omega}_\lambda\in N_\lambda$, then it is obvious that $\bar{\omega}_\lambda$ is a maximizer of $E$ on $N_\lambda$. By Theorem $\ref{22}$, $\bar{\omega}_\lambda$ must satisfy $\bar{\omega}_\lambda=\lambda I_{\{\bar{\psi}_{\lambda}>\bar{\mu}_\lambda\}}$ for some $\bar{\mu}_\lambda\in \mathbb R$ and the support of $\bar{\omega}_\lambda $ shrinks to $x_1$ as $\lambda\to+\infty$. Taking into account Theorem $\ref{27}$, we deduce that $\omega_\lambda\equiv\bar{\omega}_\lambda$ if $\lambda$ is large enough, which is the desired result.
\end{proof}

\section{Proof of Theorem 2.3}

In this section, we give the proof of Theorem $\ref{22}$ for completeness. The idea is basically from \cite{T}.
\begin{lemma}
There exists $\omega_\lambda\in N_{\lambda}$ such that $E(\omega_\lambda)=\sup_{\omega\in N_\lambda}E(\omega)$. Moreover, $\omega_\lambda=\lambda I_{\{\psi_{\lambda}>\mu_\lambda\}\cap B_r(x_1)}$ for some $\mu_\lambda >0$ depending on $\omega_\lambda$, where $\psi_\lambda=G\omega_\lambda$.
\end{lemma}
\begin{proof} First we show the existence of a maximizer, that is, there exists $\omega_\lambda\in N_{\lambda}$ such that $E(\omega_\lambda)=\sup_{\omega\in N_\lambda}E(\omega)$. Since $G(x,y)\in L^{1}(D\times D)$, we have for any $\omega\in N_\lambda$
\[E(\omega)=\frac{1}{2}\int_D\int_DG(x,y)\omega(x)\omega(y)dxdy\leq \frac{1}{2}\lambda^2\int_D\int_D|G(x,y)|dxdy\leq C\lambda^2,
\]
which implies that $E$ is bounded from above over $N_\lambda$. Let $\{\omega_n\}\subset N_\lambda$ be a sequence satisfying
 \[\lim_{n\to+\infty}E(\omega_n)= \sup_{\omega\in N_\lambda}E(\omega).\]

  Since $N_\lambda$ is a bounded set in $L^\infty(D)$, thus is sequentially compact in the weak star topology in $L^\infty(D)$. Without loss of generality, we assume that $\omega_n\rightarrow \omega_\lambda$ weakly star in $L^\infty(D)$ for some $\omega_\lambda\in L^\infty(D)$ as $n\rightarrow +\infty$.
We claim that $\omega_\lambda\in N_\lambda$. In fact, $\omega_n\rightarrow \omega_\lambda$ weakly star in $L^\infty(D)$ means
\begin{equation}\label{**}
\lim_{n\rightarrow +\infty}\int_D\omega_n\phi dx=\int_D\omega_\lambda\phi dx,\,\,\forall\, \phi\in L^1(D).
\end{equation}
For any $\phi\in C_0^\infty(D\setminus  B_r(x_1))$, by the definition of $N_\lambda$ we have
\[
\lim_{n\rightarrow +\infty}\int_D\omega_n\phi dx=\int_D\omega_\lambda\phi dx=0,
\]
which implies $supp(\omega_\lambda)\subset B_r(x_1)$. By choosing $\phi\equiv1$ in \eqref{**}, we have
\[
\int_D\omega_\lambda dx=\lim_{n\rightarrow +\infty}\int_D\omega_n dx=1.
\]
Now we prove that $0\leq \omega_\lambda\leq\lambda$ a.e. in $D$. We prove $\omega_\lambda\leq\lambda$ first. Suppose that $|\{\omega_\lambda>\lambda\}|>0$, then there exist $\varepsilon_0,\varepsilon_1>0$ such that $|\{\omega_\lambda>\lambda+\varepsilon_0\}|>\varepsilon_1$. Denote $B^*=\{\omega_\lambda>\lambda+\varepsilon_0\}\subset D$, then for $\phi=I_{B^*}$ by weak star convergence we have
\[\lim_{n\rightarrow +\infty}\int_D(\omega_\lambda-\omega_n)\phi dx=\lim_{n\rightarrow +\infty}\int_{B^*}(\omega_\lambda-\omega_n)dx=0.\]
On the other hand,
\[\int_{B^*}(\omega_\lambda-\omega_n)dx\geq\varepsilon_0\varepsilon_1,\]
which is a contradiction. So we have $\omega_\lambda\leq \lambda$ a.e. in $D$. Repeating this procedure, we obtain $\omega_\lambda\geq 0$ a.e. in $D$. Therefore we have proved $\omega_\lambda\in N_\lambda$.

Finally, by the property of weak star convergence, we have
\[\lim_{n\rightarrow +\infty}\frac{1}{2}\int_D\int_DG(x,y)\omega_n(x)\omega_n(y)dxdy=\lim_{n\rightarrow +\infty}\frac{1}{2}\int_D\int_DG(x,y)\omega_\lambda(x)\omega_\lambda(y)dxdy,\]
 which gives $E(\omega_\lambda)=\sup_{\omega\in N_\lambda}E(\omega)$. So $E$ attains its maximum over $N_\lambda.$

 Now we show that for any maximizer $\omega_\lambda$, there exists $\mu_\lambda>0$ such that $\omega_\lambda=\lambda I_{\{\psi_{\lambda}>\mu_\lambda\}\cap B_r(x_1)}$. To this end, we define a family of test functions $\omega^s(x)=\omega_\lambda+s\left(z_0(x)-z_1(x)\right)$, $s>0$, where $z_0,z_1$ satisfies
\begin{equation}
\begin{cases}
z_0,z_1\in L^\infty(D),

 \\ \int_Dz_0dx=\int_D z_1dx,
 \\ z_0,z_1\geq 0,
 \\ supp(z_0),supp(z_1)\subset B_r(x_1),
 \\z_0=0\text{\,\,\,\,\,\,} in\text{\,\,} D\verb|\|\{\omega_\lambda\leq\lambda-\delta\},
 \\z_1=0\text{\,\,\,\,\,\,} in\text{\,\,} D\verb|\|\{\omega_\lambda\geq\delta\}.
\end{cases}
\end{equation}
Here $\delta$ is a positive parameter. Note that for fixed $z_0,z_1$ and $\delta$, if $s$ is sufficiently small, then $\omega^s\in N_\lambda$. So we have
\[0\geq\frac{dE(\omega^s)}{ds}|_{s=0^+}=\int_Dz_0\psi_\lambda dx-\int_Dz_1\psi_\lambda dx,\]
where $\psi_\lambda=G\omega_\lambda$.
 This gives
\[\sup_{\{\omega_\lambda<\lambda\}\cap B_r(x_1)}\psi_\lambda\leq\inf_{\{\omega_\lambda>0\}\cap B_r(x_1)}\psi_\lambda.\]
Since $\overline{B_r(x_1)}$ is connected and $\overline{\{\omega_\lambda<\lambda\}\cap B_r(x_1)}\cup\overline{\{\omega_\lambda>0\}\cap B_r(x_1)}=\overline{B_r(x_1)}$, we have $\overline{\{\omega_\lambda<\lambda\}\cap B_r(x_1)}\cap\overline{\{\omega_\lambda>0\}\cap B_r(x_1)}\neq\varnothing$, then by the continuity of $\psi_\lambda$ we deduce that
\[\sup_{\{\omega_\lambda<\lambda\}\cap B_r(x_1)}\psi_\lambda=\inf_{\{\omega_\lambda>0\}\cap B_r(x_1)}\psi_\lambda.\]

Define \[\mu_\lambda:=\sup_{\{\omega_\lambda<\lambda\}\cap B_r(x_1)}\psi_\lambda=\inf_{\{\omega_\lambda>0\}\cap B_r(x_1)}\psi_\lambda,\]
then it is easy to check that
\begin{equation}
\begin{cases}
\omega_\lambda=0\text{\,\,\,\,\,\,a.e. in} \text{\,\,}\{\psi_\lambda<\mu_\lambda\}\cap B_r(x_1),
 \\ \omega_\lambda=\lambda\text{\,\,\,\,\,\,a.e. in } \text{\,\,}\{\psi_\lambda>\mu_\lambda\}\cap B_r(x_1).
\end{cases}
\end{equation}
On the level set $\{\psi_\lambda=\mu_\lambda\}\cap B_r(x_1)$, we have $\nabla\psi_\lambda=0\text{\,\,}$ a.e. by the property of Sobolev functions, thus $\omega_\lambda=-\Delta \psi_\lambda=0$ a.e.. To summarize, we have obtained
\begin{equation}
\begin{cases}
\omega_\lambda=0\text{\,\,\,\,\,\,a.e.} \text{\,\,}\{\psi_\lambda\leq\mu_\lambda\}\cap B_r(x_1),
 \\ \omega_\lambda=\lambda\text{\,\,\,\,\,\,a.e.} \text{\,\,}\{\psi_\lambda>\mu_\lambda\}\cap B_r(x_1),
\end{cases}
\end{equation}
or equivalently, $\omega_\lambda=\lambda I_{\{\psi_\lambda>\mu_\lambda\}}\cap B_r(x_1)$.

Finally, by taking $r$ small and $\lambda$ large we have $\mu_\lambda>0$.
\end{proof}

Now we estimate the size and location of $supp(\omega_\lambda)$ as $\lambda\rightarrow +\infty$. This is somewhat different from \cite{T}. Define $\zeta_\lambda:=\psi_\lambda-\mu_\lambda$,  $\Omega:=\{\psi_\lambda>\mu_\lambda\}\cap B_r(x_1)$ which is called the vortex core, and \[T(\omega_\lambda):= \frac{1}{2}\int_D\zeta_\lambda\omega_\lambda dx, \] which represents the kinetic energy of $\omega_\lambda$ on $\{\psi_\lambda>\mu_\lambda\}$. Note that $\{\zeta_\lambda>0\}\subset\subset D$ since $\mu_\lambda>0$.

Obviously we have the identity $E(\omega_\lambda)=T(\omega_\lambda)+\frac{1}{2}\mu_\lambda$. Moreover, integration by parts gives
\[\begin{split}
T(\omega_\lambda)
=&\frac{1}{2}\int_{\{\zeta_\lambda>0\}}\zeta_\lambda\omega_\lambda dx\\
=&\frac{1}{2}\int_{\{\zeta_\lambda>0\}}\zeta_\lambda(-\Delta\zeta_\lambda) dx\\
=&\frac{1}{2}\int_{\{\zeta_\lambda>0\}}|\nabla\zeta_\lambda|^2 dx\\
=&\frac{1}{2}\int_D|\nabla\zeta_\lambda^+|^2 dx.
\end{split}\]
Here we used $\{\zeta_\lambda>0\}\subset\subset D$.

\begin{lemma}
$T(\omega_\lambda)\leq C$, where $C$ is a positive number not depending on $\lambda$.
\end{lemma}
\begin{proof}

First we apply H\"{o}lder inequality to obtain
\[\begin{split}
T(\omega_\lambda)=&\frac{1}{2}\int_D\zeta_\lambda\omega_\lambda dx=\frac{1}{2}\lambda\int_{\Omega}\zeta_\lambda dx\leq\frac{1}{2}\lambda|\Omega|^{\frac{1}{2}}\left(\int_{\Omega}|\zeta_\lambda|^2dx\right)^{\frac{1}{2}}
=\frac{1}{2}\lambda|\Omega|^{\frac{1}{2}}\left(\int_{B_r(x_1)}|\zeta^+_\lambda|^2dx\right)^{\frac{1}{2}}.
\end{split}\]
On the other hand, by Sobolev embedding $W^{1,1}({B_r(x_1))}\hookrightarrow L^2({B_r(x_1))}$, we have
\[\begin{split}
\left(\int_{B_r(x_1)}|\zeta^+_\lambda|^2 dx\right)^{\frac{1}{2}}
\leq& C\left(\int_{B_r(x_1)}\zeta^+_\lambda dx+\int_{B_r(x_1)}|\nabla\zeta^+_\lambda|dx\right)\\
=&C\left(\int_{\Omega}\zeta^+_\lambda dx+\int_{\Omega}|\nabla\zeta^+_\lambda|dx\right).
\end{split}\]
Here and in the sequel we use $C$ to denote various positive numbers independent of $\lambda$.
Therefore
\[\begin{split}
T(\omega_\lambda)\leq&\frac{1}{2}C\lambda|\Omega|^{\frac{1}{2}}\int_{\Omega}\zeta_\lambda dx+\frac{1}{2}C\lambda|\Omega|^{\frac{1}{2}}\int_{\Omega}|\nabla\zeta_\lambda|dx\\
\leq&CT(\omega_\lambda)|\Omega|^{\frac{1}{2}}+\frac{1}{2}C\lambda|\Omega|\left(\int_{\Omega}|\nabla\zeta_\lambda|^2dx\right)^{\frac{1}{2}}\\
\leq&CT(\omega_\lambda)\lambda^{-\frac{1}{2}}+\frac{1}{2}C\left(T(\omega_\lambda)\right)^{\frac{1}{2}},
\end{split}\]
where we used $\lambda|\Omega|=\int_D\omega_\lambda=1$. By choosing $\lambda$ large enough such that $C\lambda^{-{1}/{2}}<\frac{1}{2}$, we deduce that $T(\omega_\lambda)\leq C$, which is the desired result.

\end{proof}
\begin{lemma}

$E(\omega_\lambda)\geq-\frac{1}{4\pi}\ln \varepsilon-C$, where $\varepsilon={1}/{\sqrt{\lambda\pi}}$.
\end{lemma}
\begin{proof}

Define $\bar{\omega}_\lambda=\lambda I_{B_\varepsilon(x_1)}$. It is easy to see that $\bar{\omega}_\lambda\in N_\lambda$, so we have $E(\omega_\lambda)\geq E(\bar{\omega}_\lambda)$. Now we calculate $E(\bar{\omega}_\lambda)$:

\begin{equation}\label{noon}
\begin{split}
E(\bar{\omega}_\lambda)=&\frac{1}{2}\int_D\int_DG(x,y)\bar{\omega}_\lambda(x)\bar{\omega}_\lambda(y)dxdy\\
=&-\frac{1}{4\pi}\int_D\int_D\ln|x-y|\bar{\omega}_\lambda(x)\bar{\omega}_\lambda(y)dxdy-\frac{1}{2}\int_D\int_Dh(x,y)\bar{\omega}_\lambda(x)\bar{\omega}_\lambda(y)dxdy\\
=&-\frac{\lambda^2}{4\pi}\int_{B_\varepsilon(x_1)}\int_{B_\varepsilon(x_1)}\ln|x-y|dxdy-\frac{1}{2}\int_{B_\varepsilon(x_1)}\int_{B_\varepsilon(x_1)}h(x,y)dxdy.
\end{split}
\end{equation}
Since $|x-y|\leq2\varepsilon$ for $x,y\in B_{\varepsilon}(x_1)$, we have
\begin{equation}\label{noon2}
\begin{split}
-\frac{\lambda^2}{4\pi}\int_{B_\varepsilon(x_1)}\int_{B_\varepsilon(x_1)}\ln|x-y|dxdy
\geq& -\frac{\lambda^2}{4\pi}\int_{B_\varepsilon(x_1)}\int_{B_\varepsilon(x_1)}\ln|2\varepsilon|dxdy\\
=&-\frac{1}{4\pi}\ln\varepsilon-\frac{1}{4\pi}\ln2.
\end{split}
\end{equation}
On the other hand, the integral $\int_{B_\varepsilon(x_1)}\int_{B_\varepsilon(x_1)}h(x,y)dxdy$ converges to $h(x_1,x_1)$ as $\lambda \rightarrow +\infty$, therefore is bounded, or equivalently
\begin{equation}\label{noon3}
\big|\int_{B_\varepsilon(x_1)}\int_{B_\varepsilon(x_1)}h(x,y)dxdy\big|\leq C.
\end{equation}
Taking into account \eqref{noon} \eqref{noon2} and \eqref{noon3} we get
\[E(\omega_\lambda)\geq-\frac{1}{4\pi}\ln \varepsilon-C.\]

\end{proof}

From Lemma 4.2, Lemma 4.3 and the identity $E(\omega_\lambda)=T(\omega_\lambda)+\mu_\lambda/2$, we immediately obtain

\begin{lemma}
$\mu_\lambda\geq -\frac{1}{2\pi}\ln \varepsilon- C$.
\end{lemma}

Now we show that the size of $supp(\omega_\lambda)$ is of order $\varepsilon$.
\begin{lemma}
There exists some $R_0>0$ such that $diam(supp(\omega_\lambda))<R_0\varepsilon$ when $\lambda$ is large enough.
\end{lemma}
\begin{proof}
For any $x\in supp(\omega_\lambda)$, we have by definition $\psi_\lambda(x)\geq\mu_\lambda$, that is,
\begin{equation}\label{48}
\int_DG(x,y)w_\lambda(y)dy\geq-\frac{1}{2\pi}\ln \varepsilon -C.
\end{equation}
On the other hand, since $h(x,y)$ is bounded from below in $D\times D,$ we have
\begin{equation}\label{49}
\int_DG(x,y)w_\lambda(y)dy \leq -\frac{1}{2\pi}\int_D\ln|x-y|\omega_\lambda(y)dy+C.
\end{equation}
Combining $\eqref{48}$ and $\eqref{49}$ and by simple calculation, we can easily get
\begin{equation}\label{lnn}
\int_D\ln\frac{\varepsilon}{|x-y|}\omega_\lambda(y)dy\geq C.
\end{equation}

Let $R>1$ be a number to be determined. we divide the integral on the left side of \eqref{lnn} into two parts
\begin{equation}\label{50}
\int_{B_{R\varepsilon}(x)}\ln\frac{\varepsilon}{|x-y|}\omega_\lambda(y)dy+\int_{D\verb|\|B_{R\varepsilon}(x)}\ln\frac{\varepsilon}{|x-y|}\omega_\lambda(y)dy\geq C.
\end{equation}
The first integral in $\eqref{50}$ can be estimated by the rearrangement inequality as follows
\[\int_{B_{R\varepsilon}(x)}\ln\frac{\varepsilon}{|x-y|}\omega_\lambda(y)dy\leq \lambda\int_{B_\varepsilon(x)}\ln\frac{\varepsilon}{|x-y|}dy=\frac{1}{2}.\]
So we obtain
\begin{equation}\label{dont}
\int_{D\verb|\|B_{R\varepsilon}(x)}\ln\frac{\varepsilon}{|x-y|}\omega_\lambda(y)dy \geq C.
\end{equation}
Notice that $|x-y|\geq R\varepsilon$ for any $y\in B_{R\varepsilon}(x)$, so we obtain from \eqref{dont}
\[\int_{D\verb|\|B_{R\varepsilon}(x)}\ln\frac{\varepsilon}{|x-y|}\omega_\lambda(y)dy\leq\int_{D\verb|\|B_{R\varepsilon}(x)}\ln\frac{1}{R}\omega_\lambda(y)dy,\]
or equivalently,
\[\int_{D\verb|\|B_{R\varepsilon}(x)}\omega_\lambda(y)dy\leq  \frac{C}{\ln R}.\]
Taking into account the fact $\int_D\omega_\lambda dx=1,$ we get
\[\int_{B_{R\varepsilon}(x)}\omega_\lambda(y)dy\geq 1-  \frac{C}{\ln R}.\]
Choosing $R$ large such that $ 1- \frac{C}{\ln R}>\frac{1}{2}$, we obtain
\begin{equation}\label{arb}
\int_{B_{R\varepsilon}(x)}\omega_\lambda(y)dy>\frac{1}{2}.
\end{equation}
Since $\int_D\omega_\lambda=1 dx$  and \eqref{arb} holds true for arbitrary $x\in supp(\omega_\lambda)$, we deduce that
\[diam(supp\omega_\lambda)<2R\varepsilon.\]
Thus the lemma is proved by choosing $R_0=2R$.

\end{proof}
Now we estimate the location of $supp(\omega_\lambda)$.
\begin{lemma}
 $\lim_{\lambda\to+\infty}\int_Dx\omega_\lambda(x)dx= x_1$.
\end{lemma}
\begin{proof}
Denote $x_\lambda:= \int_Dx\omega_\lambda(x)dx$, then obviously $x_\lambda\in \overline{B_r(x_1)}$. For any sequence $\{x_{\lambda_j}\},\lambda_j\rightarrow +\infty$, there exists a subsequence $\{x_{\lambda_{j_k}}\}$ such that
$x_{\lambda_{j_k}}\rightarrow z_1\in \overline{B_r(x_1)}$. For simplicity, we still denote the subsequence by $\{x_{\lambda_k}\}$.  It suffices to show that $z_1=x_1$.

Define $\bar{\omega}_\lambda:= \lambda I_{B_\varepsilon(x_1)}$. Since $ E(\bar{\omega}_\lambda)\leq E(\omega_\lambda)$, we obtain
\begin{equation}\label{ries}
\begin{split}
\int_D\int_D -\frac{1}{2\pi}\ln|x-y|\bar{\omega}_{\lambda_k}(x)\bar{\omega}_{\lambda_k}(y)dxdy-\int_D\int_Dh(x,y)\bar{\omega}_{\lambda_k}(x)\bar{\omega}_{\lambda_k}(y)dxdy \\
\leq \int_D\int_D -\frac{1}{2\pi}\ln|x-y|\omega_{\lambda_k}(x)\omega_{\lambda_k}(y)dxdy-\int_D\int_Dh(x,y)\omega_{\lambda_k}(x)\omega_{\lambda_k}(y)dxdy.
\end{split}
\end{equation}
By Riesz's rearrangement inequality (see \cite{LL}, 3.7),
\begin{equation}\label{rie}
\begin{split}
\int_D\int_D -\frac{1}{2\pi}\ln|x-y|\bar{\omega}_{\lambda_k}(x)\bar{\omega}_{\lambda_k}(y)dxdy
\geq \int_D\int_D -\frac{1}{2\pi}\ln|x-y|\omega_{\lambda_k}(x)\omega_{\lambda_k}(y)dxdy.
\end{split}
\end{equation}
Combining \eqref{ries} and \eqref{rie} we get
\begin{equation}\label{jdjd}\begin{split}
\int_D\int_D h(x,y)\bar{\omega}_{\lambda_k}(x)\bar{\omega}_{\lambda_k}(y)dxdy \geq \int_D\int_D h(x,y)\omega_{\lambda_k}(x)\omega_{\lambda_k}(y)dxdy.
\end{split}
\end{equation}
Passing $k\rightarrow +\infty$ in \eqref{jdjd}, we have $h(z_1,z_1)\leq h(x_1,x_1)$, that is, $H(z_1)\leq H(x_1)$. Since $x_1$ is the unique minimum point of $H(x)$ in $\overline{B_r(x_1)}$, we obtain $z_1=x_1$, which is the desired result.
\end{proof}

\begin{lemma}
$\omega_\lambda$ is a steady solution of the vorticity equation, that is, satisfies $\eqref{888}$, provided that $\lambda$ is large enough.
\end{lemma}
\begin{proof}
For any $\xi\in C^{\infty}_c(D)$, we define a family of $C^1$ transformations $\Phi_t(x), t\in(-\infty,+\infty)$, from $D$ to $D$ by solving the following ODE,
\begin{equation}\label{400}
\begin{cases}\frac{d\Phi_t(x)}{dt}=J\nabla\xi(\Phi_t(x)),\,\,\,t\in\mathbb R, \\
\Phi_0(x)=x,
\end{cases}
\end{equation}
where $J$ denotes clockwise rotation through ${\pi}/{2}$ as before. Note that $\eqref{400}$ is solvable for all $t$ since $J\nabla\xi$ is a smooth vector field with compact support in $D$. It is easy to see that $J\nabla\xi$ is divergence-free, so by Liouville theorem(see \cite{MP2}, Appendix 1.1), $\Phi_t(x)$ is an area-preserving transformation for any fixed $t$, that is,
\[\Phi_t(A)=A,\,\,\forall\, A\subset D.\]
 Now we define a family of test functions
\begin{equation}
\omega_t(x):=\omega_\lambda(\Phi_t(x)).
\end{equation}
It is easy to check that $\omega_t\in R_{\omega_\lambda}$. Since $supp(\omega_\lambda)$ shrinks to $x_1$ as $\lambda\to+\infty,$ so $dist(supp(\omega_\lambda, \partial B_r(x_1)))>0$ if $\lambda$ is large enough. Therefore we have $\omega_t\in N_\lambda$ if $|t|$ is small,  from which we obtain \begin{equation}\label{deri}\frac{dE(\omega_t)}{dt}\bigg|_{t=0}=0.\end{equation}
 For $|t|<<1$, we expand $E(\omega_t)$ at $t=0$ as follows
\begin{equation}\label{expan}
\begin{split}
E(\omega_t)=&\frac{1}{2}\int_D\int_DG(x,y)\omega_\lambda(\Phi_t(x))\omega_\lambda(\Phi_t(y))dxdy\\
=&\frac{1}{2}\int_D\int_DG(\Phi_{-t}(x),\Phi_{-t}(y))\omega_\lambda(x)\omega_\lambda(y)dxdy\\
=&E(\omega_\lambda)+t\int_D\omega_\lambda\nabla G\omega_\lambda\cdot J\nabla\xi dx+o(t).
\end{split}
\end{equation}
Combining \eqref{deri} and \eqref{expan} together we immediately get
\[\int_D\omega_\lambda\nabla G\omega_\lambda\cdot J\nabla\xi dx=0,\]
which completes the proof Lemma 4.6.
\end{proof}

\begin{proof}[Proof of Theorem 2.3]
It follows from Lemma 4.1, Lemma 4.4, Lemma 4.5 and Lemma 4.6.
\end{proof}

\noindent{\bf Acknowledgments:}
 Daomin Cao was supported by NNSF of China Grant (No. 11831009) and Chinese Academy of Sciences by Grant QYZDJ-SSW-SYS021. Guodong Wang was supported by NNSF of China Grant (No.11771469).

\end{document}